\documentclass[12pt]{article}

\usepackage{amsmath, amssymb, amsthm}

\theoremstyle{plain}
\newtheorem{theorem}{Theorem}
\newtheorem{lemma}[theorem]{Lemma}
\newtheorem{corollary}[theorem]{Corollary}
\newtheorem{conjecture}[theorem]{Conjecture}
\usepackage{fullpage}
\usepackage{url}
\usepackage{hyperref}

\def\modd#1 #2{#1\ \mbox{{\rm (mod}} \ #2\mbox{\rm )}}

\DeclareMathOperator{\qp}{qp}
\DeclareMathOperator{\qd}{qd}
\DeclareMathOperator{\tr}{tr}
\DeclareMathOperator{\qs}{qs}
\DeclareMathOperator{\qt}{qt}
\DeclareMathOperator{\qr}{qr}
\DeclareMathOperator{\li}{li}

\newtheorem*{qrpNOTdiscrepeat}{Lemma~\ref{qrNOTdiscpee}}

\begin{document}

\title{Quadratic sequences with prime power discriminators}

\author{Sajed Haque\\
School of Computer Science \\
University of Waterloo \\
Waterloo, ON  N2L 3G1 \\
Canada \\
{\tt s24haque@cs.uwaterloo.ca}
}

\date{\today}

\maketitle

\begin{abstract}
The discriminator of an integer sequence $\textbf{s} = (s(i))_{i \geq 0}$, introduced by Arnold, Benkoski, and McCabe in 1985, is the function $D_{\textbf{s}} (n)$ that sends $n$ to the least integer $m$ such that the numbers $s(0), s(1), \ldots, s(n - 1)$ are pairwise incongruent modulo $m$.  In this note, we try to determine all quadratic sequences whose discriminator is given by $p^{\lceil \log_p n \rceil}$ for prime $p$, i.e., the smallest power of $p$ which is $\geq n$. We determine all such sequences for $p = 2$, show that there are none for $p \geq 5$, and provide some partial results for $p = 3$.
\end{abstract}

\section{Discriminators}

Let $S$ be a set of integers. If there is some integer $m$ for which the elements of $S$ are pairwise incongruent modulo $m$, we say that $m$ \textit{discriminates} $S$. Now let $\textbf{s} = (s(i))_{i \geq 0}$ be a
sequence of distinct integers. For all integers $n \geq 1$, we define $D_\textbf{s} (n)$ to be the least positive integer $m$ that discriminates the set $\{s(0), s(1), \ldots, s(n - 1)\}$.  The function $D_\textbf{s} (n)$ is called the \textit{discriminator} of the sequence \textbf{s}. 

The discriminator was first introduced by Arnold, Benkoski, and McCabe \cite{ABM85}. They derived the discriminator for the sequence $1, 4, 9, \ldots$ of positive integer squares. More recently, discriminators of various sequences were studied by Schumer and Steinig \cite{SchumerSteinig88}, Barcau \cite{Barcau88}, Schumer \cite{Schumer90}, Bremser, Schumer, and Washington \cite{BSW90},
Moree and Roskam \cite{MoreeRoskam95}, Moree \cite{Moree96}, Moree and Mullen \cite{MoreeMullen96}, Zieve \cite{Zieve98}, Sun \cite{Sun13},
Moree and Zumalac\'arrequi \cite{MoreeZumalacarregui16}, Haque and Shallit \cite{HaqueShallit16}, Ciolan and Moree \cite{CiolanMoree19}, Haque and Shallit \cite{HaqueShallit19}, and Faye, Luca, and Moree \cite{FLM19}.

In most cases, different sequences also had different discriminators. However, Sun \cite{Sun13} showed that the triangular numbers $(k(k - 1)/2)_{k \geq 1}$ and the sequence $(k(2k - 1))_{k \geq 0}$ both have discriminators given by $2^{\lceil \log_2 n \rceil}$ for the first $n$ terms. Furthermore, Haque and Shallit showed that the odious numbers also have the discriminator given by $2^{\lceil \log_2 n \rceil}$ \cite{HaqueShallit16}, and later presented a class of exponential sequences with this same discriminator \cite{HaqueShallit19}.

On a related note, the Salajan sequence was shown, by Moree and Zumalac\'arrequi to be discriminated by $2^{\lceil \log_2 n \rceil}$ and $5^{\lceil \log_5 (5n/4)\rceil}$, with the discriminator being the minimum of the two \cite{MoreeZumalacarregui16}. Haque and Shallit showed that the discriminator for the evil numbers is $2^{\lceil \log_2 n \rceil}$, except when $n = 2^i + 1$ for some $i \geq 2$. 

Sun also showed that $k (3k - 1)_{k \geq 0}$ has discriminator $3^{\lceil \log_3 n \rceil}$ \cite{Sun13}.

In this paper, we explore the quadratic sequences whose discriminators are given by $p^{\lceil \log_p n \rceil}$ for some fixed prime $p$. We provide a complete characterization of such sequences for the case of $p = 2$, show that there are no such sequences for $p \geq 5$, and provide some partial results for $p = 3$.

The outline of the paper is as follows. 
Section~\ref{two} describes the general approach for proving the results, with useful observations and lemmas. 
Section~\ref{three} characterizes the case of $p = 2$, fully describing all quadratic sequences with discriminator $2^{\lceil \log_2 n \rceil}$. 
Section~\ref{four} considers the case of $p \geq 5$, showing that there are no quadratic sequences with discriminator $p^{\lceil \log_p n \rceil}$ for $p \geq 5$.
Finally, Section~\ref{five} explores the case of $p = 3$, with a set of necessary conditions and a set of sufficient conditions for quadratic sequences to have discriminator $3^{\lceil \log_3 n \rceil}$.

\section{Approach}
\label{two}

We denote quadratic sequences by $(q (n))_{n \geq 0} = (\alpha n^2 + \beta n + \gamma)_{n \geq 0}$, for rational numbers $\alpha$, $\beta$, and $\gamma$. Our approach involves exploiting the property that for any $n \geq 1$, the discriminator $D_q (n)$ is the smallest integer that does not divide $q(j) - q(i)$ for all pairs of integers $i$ and $j$ such that $0 \leq i < j < n$. Here we can see that
\begin{align}
	q (j) - q (i) &= \alpha j^2 + \beta j - \alpha i^2 - \beta i = \alpha (j^2 - i^2) + \beta (j - i) \nonumber \\
	&= (j - i)(\alpha(i + j) + \beta).
	\label{quaddiff}
\end{align}
Eq.~\eqref{quaddiff} is used to prove various results in later sections. Some of these results are used to show that some quadratic sequences have discriminator $D_q (n) = p^{\lceil \log_p n \rceil}$ for some prime $p$. This is accomplished by applying the following lemma:

\begin{lemma}
	Let $p \geq 2$ be a prime number and let $(s (n))_{n \geq 0}$ be a sequence of distinct integers that satisfies the following conditions:
	\begin{enumerate}
		\item For all pairs of integers $k$ and $m$ such that $k \geq 0$ and $0 \leq m < p^{k + 1}$, there exists a pair of integers $i$ and $j$ such that $0 \leq i < j \leq p^k$ and $m | s (j) - s (i)$;
		\item For all integers $k$, $i$, and $j$ such that $k \geq 0$ and $0 \leq i < j < p^{k + 1}$, we have $p^{k + 1} \nmid s (j) - s(i)$. 
	\end{enumerate}
	Then $D_s (n) = p^{\lceil \log_p n \rceil}$ for $n > 0$.
	\label{generaldisc}
\end{lemma}

\begin{proof}
	The case $n = 1$ follows from the fact that $D_s (1) = 1$ regardless of the given conditions. Otherwise, let $k \geq 0$ be such that $p^k < n \leq p^{k + 1}$. From the first condition, we know that for all $0 \leq m < p^{k + 1}$, there exists a pair of integers $i$ and $j$ such that $0 \leq i < j \leq p^k \leq n - 1$ and $m | s (j) - s (i)$. This means that $m$ does not discriminate the set 
	\begin{equation*}
		\{s (0), s (1), \ldots, s (n - 1)\}, 
	\end{equation*}
	and thus, $D_s (n) \geq p^{k + 1}$ for all $p^k < n \leq p^{k + 1}$.
	
	Furthermore, from the second condition, we know that $p^{k + 1} \nmid s (j) - s (i)$ as long as $0 \leq i < j < p^{k + 1}$. So for $p^k < n \leq p^{k + 1}$, we know $p^{k + 1}$ cannot divide $s (j) - s (i)$ for all $i$, $j$ in the range $0 \leq i < j \leq n - 1$ since $n - 1 < p^{k + 1}$. Therefore, $D_s(n) = p^{k + 1} = p^{\lceil \log_p n \rceil}$.
\end{proof}

Other results in later sections involve scenarios in which $D_s (n) \neq p^{\lceil \log_p n \rceil}$ for some prime $p$ and integer $n \geq 1$. This is achieved with the help of the following lemma:

\begin{lemma}
	Let $p \geq 2$ be a prime number and let $(q (n))_{n \geq 0} = (\alpha n^2 + \beta n + \gamma)_{n \geq 0}$ be a quadratic sequence such that $\alpha$, $\beta$, and $\gamma$ are integers that satisfy any of the following conditions:
	\begin{enumerate}
		\item $p \nmid \alpha$;
		\item $p | \beta$;
		\item $\alpha = p^k c$ for some integer $c$ such that $c \nmid \beta$.
	\end{enumerate}
	Then there exists a value of $n > 0$ such that $D_q (n) \neq p^{\lceil \log_p n \rceil}$.
	\label{generalNOTdisc}
\end{lemma}

\begin{proof}
	We can assume $\gamma = 0$ since the discriminator does not depend on $\gamma$. We now consider the conditions one by one, while recalling from Eq.~\eqref{quaddiff} that $q (j) - q (i) = (j - i)(\alpha(i + j) + \beta)$ for all pairs of integers $i$ and $j$.
	
	\begin{description}
		\item[Case 1: $p \nmid \alpha$.] For any integer $\ell \geq 2$, we show that there exists a pair of $i$ and $j$ such that $0 \leq i < j < p^\ell$ and $p^\ell | q(j) - q(i)$. Since $p \nmid \alpha$, this implies that $\alpha$ and $p^\ell$ are co-prime. We choose $i = 0$ and $j = -\beta(\alpha)^{-1} \bmod p^\ell$. Then
		\begin{align*}
			q(j) - q(i) = j(\alpha j + \beta) &\equiv \modd{j(\alpha (-\beta)(\alpha)^{-1} + \beta)} {p^\ell}\\
			&\equiv j(-\beta + \beta) \equiv \modd{0} {p^\ell}.
		\end{align*}
		Since $p^\ell | q(j) - q(i)$ while $0 \leq i < j < p^\ell$, it follows that $D_{q} (p^\ell) \neq p^\ell$ and so, $D_q (n) \neq p^{\lceil \log_p n \rceil}$ for $n = p^\ell$.
		\item[Case 2: $p | \alpha$ and $p | \beta$.] Again, for any integer $\ell \geq 2$, we show that there exists a pair of $i$ and $j$ such that $0 \leq i < j < p^\ell$ and $p^\ell | q(j) - q(i)$. Here, we choose $i = 0$ and $j = p^{\ell - 1}$ to get
		\begin{equation*}
			q(j) - q(i) = (j - i)(\alpha(i + j) + \beta) = p^{\ell - 1}(p^{\ell - 1}\alpha + \beta) = p^\ell(p^{\ell - 2}\alpha + \frac{\beta}{p}),
		\end{equation*}
		noting that $p^{\ell - 2}$ and $\frac{\beta}{p}$ are integers. Just as with Case 1, this implies that $D_{q} (p^\ell) \neq p^\ell$ and so $D_q (n) \neq p^{\lceil \log_p n \rceil}$ for $n = p^\ell$.
		\item[Case 3: $p | \alpha$, $p \nmid \beta$, but $c \nmid \beta$, where $\alpha = p^k c$ for $p \nmid c$.] Let $r$ be any prime number such that $r | c$ and $r \nmid \beta$. For any pair of integers $i$ and $j$ such that $0 \leq i < j < r$, we have $q(j) - q(i) = (j - i)(\alpha(i + j) + \beta)$. Since $j < r$, we have $r \nmid (j - i)$. We also have $r \nmid \alpha(i + j) + \beta$ since $r | \alpha$ but $r \nmid \beta$. Therefore, $r \nmid q(j) - q(i)$. It follows that $D_q (r) \leq r < p^{\lceil \log_p r \rceil}$, i.e., $D_q (n) \neq p^{\lceil \log_p n \rceil}$ for $n = r$.
	\end{description}
	In all cases, we have $D_q (n) \neq p^{\lceil \log_p n \rceil}$ for at least one value of $n \geq 1$.
\end{proof}

Furthermore, to help generalize some results, we cite a useful lemma from Haque and Shallit \cite{HaqueShallit19}.

\begin{lemma}
	Given a sequence $s(0), s(1), \ldots$ and a non-zero integer $c$, let $s'(0), s'(1), \ldots$ denote the sequence such that $s'(i) = c s(i)$ for all $i \geq 0$. Then we have $D_{s'} (n) = D_s (n)$ for every $n$ such that $\gcd(|c|, D_s (n)) = 1$.
	\label{coprimescaledisc}
\end{lemma}

Finally, we also consider the cases in which the quadratic coefficients are not integers. Discriminators are only applicable to integer sequences, so we are interested in quadratic polynomials that are \textit{integer-preserving} i.e. polynomials such that $q(n)$ is an integer if $n$ is an integer. From a result of P\'{o}lya \cite{Polya1915}, we deduce that every integer-preserving quadratic polynomial can be written in the form
\begin{equation*}
	q(n) = c_1 \frac{n(n - 1)}{2!} + c_2 n + c_3 (1) = \frac{c_1}{2} n^2 + \frac{2 c_2 - c_1}{2} n + c_3,
\end{equation*}
for integers $c_1$, $c_2$, and $c_3$. Note that if $c_1$ is even, then $2 c_2 - c_1$ must also be even, and vice versa. Thus, we can express all integer-valued quadratic polynomials in the form
\begin{equation*}
	q (n) = \frac{\alpha'}{2} n^2 + \frac{\beta'}{2} n + \gamma,
\end{equation*}
for integers $\alpha'$, $\beta'$, and $\gamma$, where $\alpha'$ and $\beta'$ are either both even, or both odd. We denote this sequence as $(\qr (n))_{n \geq 0} = \left(\frac{\alpha'}{2} n^2 + \frac{\beta'}{2} n + \gamma\right)_{n \geq 0}$. Here, for any integers $i$ and $j$, we have
\begin{equation}
	\qr (j) - \qr (i) = (j - i)(\alpha(i + j) + \beta) = \frac{(j - i)(\alpha' (i + j) + \beta')}{2}
	\label{qrdiff}
\end{equation}

We can extend Lemma~\ref{generalNOTdisc} to apply to quadratic sequences with rational coefficients, except with odd primes $p \geq 3$ instead of $p \geq 2$.

\begin{lemma}
	Let $p \geq 3$ be a prime number and let $(\qr (n))_{n \geq 0} = \left(\frac{\alpha'}{2} n^2 + \frac{\beta'}{2} n + \gamma\right)_{n \geq 0}$ for integers $\alpha'$, $\beta'$, and $\gamma$ be a quadratic sequence such that $\alpha'$ and $\beta'$ are odd, and any of the following conditions are satisfied:
	\begin{enumerate}
		\item $p \nmid \alpha'$;
		\item $p | \beta'$;
		\item $\alpha' = p^k c$ for some integer $c$ such that $c \nmid \beta'$.
	\end{enumerate}
	Then there exists a value of $n > 0$ such that $D_{\qr} (n) \neq p^{\lceil \log_p n \rceil}$.
	\label{qrNOTdisc}
\end{lemma}

\begin{proof}
	The argument is identical to Lemma~\ref{generalNOTdisc}. For the first two cases, since $p$ is odd, it follows that $p | (j - i)(\alpha' (i + j) + \beta')$ implies $p | \frac{(j - i)(\alpha' (i + j) + \beta')}{2} = \qr(j) - \qr (i)$. For the third case, it is clear that $r \nmid (j - i)(\alpha' (i + j) + \beta')$ implies $r \nmid \frac{(j - i)(\alpha' (i + j) + \beta')}{2} = \qr(j) - \qr (i)$. Thus the same arguments apply.
\end{proof}

\section{The case $p = 2$}
\label{three}

In this section, we provide a complete characterization of all integer-valued quadratic sequences with discriminator $2^{\lceil \log_2 n \rceil}$. These quadratic sequences can be divided into two types, based on whether the quardatic coefficients are integers or not, i.e., whether $\alpha$ and $\beta$ are integers for sequences of the form $(q (n))_{n \geq 0} = (\alpha n^2 + \beta n + \gamma)_{n \geq 0}$.

\subsection{Integer quadratic coefficients}
First we focus on the case in which $\alpha$ and $\beta$ are integers. We begin by considering quadratic sequences of the form $(\qd (n))_{n \geq 0} = (2^t n^2 + bn)_{n \geq 0}$ for an integer $t > 0$ and odd integer $b$, and then extend the result later. For these sequences, we can apply Eq.~\eqref{quaddiff} to get
\begin{equation}
	\qd (j) - \qd (i) = (j - i)(2^t (i + j) + b).
\end{equation}

We compute the discriminator for $(\qd (n))_{n \geq 0}$ using Lemma~\ref{generaldisc} for $p = 2$. The first condition for Lemma~\ref{generaldisc} is established by another lemma:

\begin{lemma}
	Let $k \geq 0$. For all positive integers $m < 2^{k + 1}$, there exists at least one pair of integers, $i$ and $j$, such that $0 \leq i < j \leq 2^k$ and $m | \qd (j) - \qd(i)$.
	\label{qdlowerbound}
\end{lemma}

\begin{proof}
	We consider the different possible cases for the value of $m$.
	\begin{enumerate}
		\item $m$ is a power of 2, i.e., $m = 2^{\ell}$ where $\ell \leq k$. Set $i = 0$ and $j = 2^i$ so that $j - i = 2^\ell$.
		\item $m$ is odd. Since $m$ and $2^t$ are co-prime, this implies that $2^t$ has a multiplicative inverse modulo $m$. Let $x = -b(2^t)^{-1} \bmod m$. If $x \leq 2^k$, then we choose $i = 0$ and $j = x$. Otherwise, if $x > 2^k$, we choose $j = 2^k$ and $i = x - 2^k$. Since $x < m < 2^{k + 1}$, it follows that $i < 2^{k + 1} - 2^k = 2^k = j$. In both cases, we have $i + j = x$, and therefore,
		\begin{align*}
			\qd(j) - \qd(i) = (j - i)(2^t x + b) &\equiv \modd{(j - i)(2^t (-b)(2^t)^{-1} + b)} {m}\\
			&\equiv (j - i)(-b + b) \equiv \modd{0} {m}.
		\end{align*}
		\item $m$ is even, but not a power of 2. In this case, we can write $m = 2^\ell \cdot r < 2^{k + 1}$ for $0 < \ell < k$, and odd $r > 2$. This implies $r < 2^{k + 1 - \ell}$. This time, let $x = -b(2^{t + 1})^{-1} \bmod r$, and choose $i = (x - 2^{\ell - 1}) \bmod r$ and $j = i + 2^\ell$, to get
		\begin{align*}
			\qd (j) - \qd (i) &= (i + 2^\ell - i)(2^t (i + i + 2^\ell) + b) = 2^\ell (2^t (2i + 2^\ell) + b)\\
			&= 2^\ell (2^{t + 1} (i + 2^{\ell - 1}) + b),
		\end{align*}
		which is divisible by $2^\ell$. Also,
		\begin{align*}
			\qd (j) - \qd (i) &= 2^\ell (2^{t + 1} (i + 2^{\ell - 1}) + b) \equiv
			\modd{2^\ell (2^{t + 1} x + b)} {r}\\
			&\equiv 2^\ell (2^{t + 1}(-b)(2^{t + 1})^{-1} + b) \equiv 2^\ell (-b + b) \equiv \modd{0} {r}.
		\end{align*}
		
		We now verify the conditions on $i$ and $j$. It is clear that $0 \leq i < j$. Furthermore, we have $i < r < 2^{k + 1 - \ell}$ and $0 < \ell < k$. For $\ell = 1$, we have $j = i + 2 \leq (r - 1) + 2 = r + 1 \leq (2^{k + 1 - 1} - 1) + 1 = 2^k$. For $\ell > 1$, we have $r < 2^{k + 1 - 2} = 2^{k - 1}$, and so, $j = i + 2^\ell < r + 2^\ell < 2^{k - 1} + 2^\ell \leq 2^{k - 1} + 2^{k - 1} = 2^k$. Therefore, the condition $0 \leq i < j \leq 2^k$ is fulfilled and $\qd(j) - \qd(i)$ is divisible by both $2^\ell$ and $r$, and thus $m | \qd(j) - \qd(i)$.
	\end{enumerate}
	In all cases, we have $m | \qd (j) - \qd (i)$ for some $i$ and $j$ in the required range.
\end{proof}

The second condition of Lemma~\ref{generaldisc} is also satisfied, as shown by the next lemma:
\begin{lemma}
	Let $k \geq 0$. For all integers $i$ and $j$ satisfying $0 \leq i < j < 2^{k + 1}$, we have $2^{k + 1} \nmid \qd (j) - \qd (i)$.
	\label{qdupperbound}
\end{lemma}
\begin{proof}
	We know $\qd (j) - \qd (i) = (j - i)(2^t (i + j) + b)$, where $t > 0$. Here, the second factor is the sum of an even number and an odd number, and therefore must itself be odd and not divisible by 2. Therefore, any powers of 2 that divide $\qd (j) - \qd (i)$ must divide the first factor, $(j - i)$. But $j - i \leq j < 2^{k + 1}$. Therefore, $2^{k + 1} \nmid \qd (j) - \qd (i)$ for all $i$ and $j$ in the range $0 \leq i < j < 2^{k + 1}$.
\end{proof}

Therefore, the two conditions in Lemma~\ref{generaldisc} are satisfied for $(\qd (n))_{n \geq 0} = (2^t n^2 + bn)_{n \geq 0}$ for every integer $t > 0$ and odd integer $b$. It follows from Lemma~\ref{generaldisc} that $D_{\qd} (n) = 2^{\lceil \log_2 n \rceil}$ for $n > 0$. Along with Lemma~\ref{generalNOTdisc} and Lemma~\ref{coprimescaledisc}, this is sufficient to characterize all quadratic sequences with integer coefficients that have discriminator $2^{\lceil \log_2 n \rceil}$, as shown in the following theorem:

\begin{theorem}
	\label{qdisctwo}
	For all quadratic sequences with integer coefficients, i.e., $(q(n))_{n \geq 0} = (\alpha n^2 + \beta n + \gamma)_{n \geq 0}$ for integers $\alpha$, $\beta$, and $\gamma$, the discriminator $D_q (n)$ is equal to $2^{\lceil \log_2 n \rceil}$ for all $n \geq 0$ if and only if all of the following conditions are satisfied:
	\begin{enumerate}
		\item $\alpha$ is even, i.e., $\alpha = 2^t \cdot r$ for some $t \geq 1$ and odd $r$;
		\item $\beta$ is odd;
		\item $r | \beta$.
	\end{enumerate}
\end{theorem}

\begin{proof}
	We assume $\gamma = 0$ since the discriminator does not depend on it. Now suppose conditions (1)-(3) hold. Then the resulting sequence, $(q(n))_{n \geq 0} = (2^t rn^2 + \beta n)_{n \geq 0}$ is equivalent to the sequence $(r \cdot \qd(n))_{n \geq 0} = (r(2^t n^2 + bn))_{n \geq 0}$ with $b = \frac{\beta}{r}$. For $r = 1$, we know $D_{\qd} (n) = 2^{\lceil \log_2 n \rceil}$ by an application of Lemma~\ref{generaldisc} for $p = 2$, with Lemmas~\ref{qdlowerbound} and~\ref{qdupperbound} verifying that the conditions are fulfilled. 
	
	Since $r$ is odd, it is co-prime to $D_{\qd}$ for all $n \geq 1$, and so we can apply Lemma~\ref{coprimescaledisc} to show that $D_q (n) = D_{\qd} (n) = 2^{\lceil \log_2 n \rceil}$. 
	
	For the other direction, we observe that the violation of any one of these conditions implies the violation of a corresponding condition of Lemma~\ref{generalNOTdisc} for $p = 2$, which showed that there exists a value of $n \geq 1$ for which $D_q (n) \neq 2^{\lceil \log_2 n \rceil}$. 
\end{proof}

Thus, we have provided a complete characterization of quadratic sequences with integer coefficients that have discriminator $2^{\lceil \log_2 n \rceil}$. We can further show that the discriminator for these sequences are shift-invariant. As defined in the previous chapter, the discriminator of $(q(n))_{n \geq 0}$ is said to be \textit{shift-invariant} if it shares the same discriminator as $(q(n + c))_{n \geq 0}$ for all $c \geq 0$.

\begin{theorem}
	For quadratic sequences $(q(n))_{n \geq 0} = (\alpha n^2 + \beta n + \gamma)_{n \geq 0}$ for integers $\alpha$, $\beta$, and $\gamma$ with discriminator $D_q (n) = 2^{\lceil \log_2 n \rceil}$, the discriminator of the shifted sequence, $(\qs(n, c))_{n \geq 0} = q(n + c)$ for any integer $c$ also satisfies $D_{\qs} (n) = 2^{\lceil \log_2 n \rceil}$.
\end{theorem}
\begin{proof}
	From Theorem~\ref{qdisctwo}, we know that $\alpha = 2^t \cdot r$ for some $t \geq 1$ and odd $r$, $\beta$ is odd, and that $r | \beta$. Now, for any integer $c$, we have
	\begin{align*}
		\qs (n, c) &= q(n + c) = \alpha (n + c)^2 + \beta (n + c) + \gamma\\
		&= \alpha n^2 + 2 \alpha nc + \alpha c^2 + \beta n + \beta c + \gamma\\
		&= \alpha n^2 + (2 \alpha nc + \beta) n + (\alpha c^2 + \beta c + \gamma).
	\end{align*}
	The coefficient of $n^2$ is $\alpha = 2^t \cdot r$, which is even, while the coefficient of $n$ is $2 \alpha nc + \beta$, which is odd. Furthermore, since $r | \alpha$ and $r | \beta$, we have $r | 2 \alpha nc + \beta$. Therefore, the three conditions in Theorem~\ref{qdisctwo} are fulfilled by $(\qs(n, c))_{n \geq 0}$ and so $D_{\qs} (n) = 2^{\lceil \log_2 n \rceil}$.
\end{proof}

\subsection{Half-integer quadratic coefficients}

We now consider quadratic sequences of the form $(\qr (n))_{n \geq 0} = \left(\frac{\alpha'}{2} n^2 + \frac{\beta'}{2} n + \gamma\right)_{n \geq 0}$ for odd $\alpha'$ and $\beta'$. Recall from Eq.~\eqref{qrdiff} that
\begin{equation*}
	\qr (j) - \qr (i) = (j - i)(\alpha(i + j) + \beta) = \frac{(j - i)(\alpha' (i + j) + \beta')}{2}. \tag{\ref{qrdiff} revisited}
\end{equation*}

To characterize the discriminator of $(\qr (j))_{n \geq 0}$, we first consider the sequence of triangular numbers, $(\tr(n))_{n \geq 0} = \left(\frac{1}{2} n^2 + \frac{1}{2} n\right)_{n \geq 0}$ and extend the result to all $(\qr (j))_{n \geq 0}$. The discriminator for the sequence of triangular numbers was already shown by Sun \cite{Sun13} to be $2^{\lceil \log_2 n \rceil}$, but here we present an alternate proof that utilizes Lemma~\ref{generaldisc}.

For the sequence of triangular numbers, Eq.~\eqref{qrdiff} becomes
\begin{equation*}
	\tr(j) - \tr(i) = \frac{(j - i)(i + j + 1)}{2}.
\end{equation*}
The first condition for Lemma~\ref{generaldisc} is established by the following lemma:

\begin{lemma}
	Let $k \geq 0$. For all positive integers $m < 2^{k + 1}$, there exists at least one pair of integers, $i$ and $j$, such that $0 \leq i < j \leq 2^k$ and $m | \tr (j) - \tr (i)$.
	\label{trlowerbound}
\end{lemma}
\begin{proof}
	We consider the different possible cases for the value of $m$.
	\begin{enumerate}
		\item $m$ is a power of 2, i.e., $m = 2^\ell$ where $\ell \leq k$. Set $i = 2^\ell - 1$ and $j = 2^\ell$ to get
		\begin{align*}
			\tr (j) - \tr (i) = \frac{(2^\ell - 2^\ell + 1)(2^\ell - 1 + 2^\ell + 1)}{2} = \frac{2^{\ell + 1}}{2} = 2^\ell = m.
		\end{align*}
		\item $m$ is odd. For $m = 1$, we set $i = 0$ and $j = 1$. Otherwise, we set $i = \lfloor \frac{m}{2} \rfloor - 1$ and $j = \lceil \frac{m}{2} \rceil$ to get
		\begin{align*}
			\tr (j) - \tr (i) = \frac{(\lceil \frac{m}{2} \rceil - \lfloor \frac{m}{2} \rfloor + 1)(\lfloor \frac{m}{2} \rfloor - 1 + \lceil \frac{m}{2} \rceil + 1)}{2} = \frac{2m}{2} = m.
		\end{align*}
		\item $m$ is even, but not a power of 2. In this case, we can write $m = 2^\ell (2r + 1) < 2^{k + 1}$ for $0 < \ell < k$ and $r > 0$. This implies $r < 2^{k - \ell} \leq 2^{k - 1}$. We have two further cases here. If $r \geq 2^\ell$, set $i = r - 2^\ell$ and $j = r + 2^\ell$ to get
		\begin{align*}
			\tr (j) - \tr (i) &= \frac{(r + 2^\ell - r + 2^\ell)(r - 2^\ell + r + 2^\ell + 1)}{2}\\
			&= \frac{(2^{\ell + 1})(2r + 1)}{2} = 2^\ell (2r + 1)= m.
		\end{align*} 
		Otherwise, if $r < 2^\ell$, set $i = 2^\ell - r - 1$ and $j = r + 2^\ell$ to get
		\begin{align*}
			\tr (j) - \tr (i) &= \frac{(r + 2^\ell - 2^\ell + r + 1)(2^\ell - r - 1 + r + 2^\ell + 1)}{2}\\
			&= \frac{(2r + 1)(2^{\ell + 1})}{2} = 2^\ell (2r + 1)= m.
		\end{align*} 
		In both cases, it is clear that $0 \leq i < j$. Furthermore, since $r < 2^{k - \ell}$ and $0 < \ell < k$, we have $j = r + 2^\ell < 2^{k - 1} + 2^\ell \leq 2^{k - 1} + 2^{k - 1} = 2^k$, thus fulfilling the required constraints on $i$ and $j$.
	\end{enumerate}
	In all cases, we have $\tr (j) - \tr (i) = m$ for some $i$ and $j$ in the required range.
\end{proof}

The second condition of Lemma~\ref{generaldisc} is established by the following lemma.
\begin{lemma}
	Let $k \geq 0$. For all pairs of integers $i$ and $j$ satisfying $0 \leq i < j < 2^{k + 1}$, we have $2^{k + 1} \nmid \tr (j) - \tr (i)$.
	\label{trupperbound}
\end{lemma}
\begin{proof}
	If, for some $i$ and $j$, we have $2^{k + 1} | \tr(j) - \tr(i) = \frac{(j - i)(i + j + 1)}{2}$, it follows that $2(2^{k + 1}) = 2^{k + 2} | (j - i)(i + j + 1)$. Note that between the factors $(j - i)$ and $(i + j + 1)$, one of them must be odd while the other is even. Therefore, at most only one of those factors can be a multiple of $2^{k + 2}$ for any $i$ and $j$. However, if $0 \leq i < j < 2^{k + 1}$, then $j - i < i + j + 1 \leq 2j < 2(2^{k + 1}) = 2^{k + 2}$, and therefore, $2^{k + 2}$ cannot divide either of those factors. In other words, $2^{k + 1} \nmid \tr (j) - \tr (i)$ for all $i$ and $j$ such that $0 \leq i < j < 2^{k + 1}$.
\end{proof}

Thus, the two conditions of Lemma~\ref{generaldisc} are established for the sequence of triangular numbers through Lemma~\ref{trlowerbound} and Lemma~\ref{trupperbound} respectively. Therefore, we can apply Lemma~\ref{generaldisc} to show that $D_{\tr}(n) = 2^{\lceil \log_2 n \rceil}$ for $n > 0$. 

We now characterize all sequences of the form $(\qr (n))_{n \geq 0} = \left(\frac{\alpha'}{2} n^2 + \frac{\beta'}{2} n + \gamma\right)_{n \geq 0}$ for odd integers $\alpha'$ and $\beta'$, and any integer $\gamma$, that have the discriminator $D_{\tr}(n) = 2^{\lceil \log_2 n \rceil}$.

\begin{theorem}
	For all quadratic sequences of the form $(\qr (n))_{n \geq 0} = \left(\frac{\alpha'}{2} n^2 + \frac{\beta'}{2} n + \gamma\right)_{n \geq 0}$ for odd integers $\alpha'$ and $\beta'$, and any integer $\gamma$, the discriminator $D_{\qr} (n)$ is equal to $2^{\lceil \log_2 n \rceil}$ for all $n \geq 0$ if and only if $\alpha' = \beta'$.
	\label{qrdisctwo}
\end{theorem}
\begin{proof}
	We know $D_{\tr} (n) = 2^{\lceil \log_2 n \rceil}$ from applying Lemma~\ref{generaldisc}. Furthermore, since $\alpha$ is odd, we have $\gcd(|\alpha|, D_{\tr} (n)) = 1$ for all $n \geq 0$. Therefore, we can apply Lemma~\ref{coprimescaledisc} to show that $D_{\qr} (n) = D_{\tr} (n) = 2^{\lceil \log_2 n \rceil}$.
	
	On the other hand, if $\alpha' \neq \beta'$, we have two cases to consider:
	\begin{description}
		\item[Case 1: $|\alpha'| \neq |\beta'|$.] Let $k$ be any integer such that $2^k > |\alpha'|$ and $2^k > |\beta'|$. We show that there exists a pair of $i$ and $j$ such that $0 \leq i < j < 2^k$ and $2^k | \qr(j) - \qr(i)$. Let $x = -\beta'(\alpha')^{-1} \bmod 2^{k + 1}$. If $x < 2^k$, we choose $i = 0$ and $j = x$. Otherwise, we choose $j = 2^k - 1$ and $i = x - j$. 
		In both cases, we have $i + j = x$, and so,
		\begin{equation*}
			\alpha'(i + j) + \beta' \equiv \alpha'(-\beta')(\alpha')^{-1} + \beta') \equiv -\beta' + \beta' \equiv \modd{0} {2^{k + 1}},
		\end{equation*}
		which implies that $2^k | \frac{\alpha'(i + j) + \beta'}{2}$ and therefore, $2^k | \qr (j) - \qr (i) = \frac{(j - i)(\alpha'(i + j) + \beta')}{2}$.
		
		It is clear that $j < 2^k$ and $i \geq 0$ for both cases, and that $i < j$ if $x < 2^k$. We now verify that $i < j$ for $x \geq 2^k$. Since $|\alpha'| \neq |\beta'|$, and both $|\alpha'|$ and $|\beta'|$ are less than $2^k$, it follows that $\alpha' \not\equiv \pm \modd{\beta'} {2^{k + 1}}$. Therefore, $x = -\beta'(\alpha')^{-1} \bmod 2^{k + 1} \not\equiv \modd{\pm 1} {2^{k + 1}}$, and so, $x < 2^{k + 1} - 1$. Also, since both $\alpha'$ and $\beta'$ are odd while $2^{k + 1}$ is even, it follows that $x$ is odd and thus, $x \leq 2^{k + 1} - 3$. Therefore, $i = x - j \leq 2^{k + 1} - 3 - 2^k + 1 = 2^k - 2 < j$.
		
		Since $2^k | \qr(j) - \qr(i)$ for some $i$ and $j$ such that $0 \leq i < j < 2^k$, we have $D_{\qr} (2^k) \neq 2^k$, and therefore, $D_{\qr} (n) \neq 2^{\lceil \log_2 n \rceil}$ for $n = 2^k$.
		\item[Case 2: $\alpha' = -\beta'$.] In this case, we have $\qr(0) = \qr(1) = 0$, and so, the sequence cannot even be discriminated.
	\end{description}
	Therefore, we have $D_{\qr} (n) = 2^{\lceil \log_2 n \rceil}$ if and only if $\alpha' = \beta'$.
\end{proof}

Unlike the case with integer coefficients, the discriminators for sequences of the form $(\qr(n))_{n \geq 0} = \left(\frac{\alpha'}{2} n^2 + \frac{\alpha'}{2} n + \gamma\right)_{n \geq 0}$ are not shift-invariant. This is because for any integer $c$, we have
\begin{align*}
	\qr (n + c) &= \frac{\alpha'}{2} (n + c)^2 + \frac{\alpha'}{2} (n + c) + \gamma = \frac{\alpha'}{2} (n^2 + 2nc + c^2) + \frac{\alpha'}{2} (n + c) + \gamma\\
	&= \frac{\alpha'}{2} n^2 + \left(\frac{2 \alpha' c}{2} + \frac{\alpha'}{2}\right) n + \left(\frac{\alpha'}{2} c^2 + \frac{\alpha'}{2} c + \gamma\right).
\end{align*}
where the coefficient of $n^2$ is $\frac{\alpha'}{2}$ while the coefficient of $n$ is $\frac{2 \alpha' c}{2} + \frac{\alpha'}{2} \neq \frac{\alpha}{2}$. Therefore, by Theorem~\ref{qrdisctwo}, the discriminator for $(\qr(n + c)_{n \geq 0}$ cannot be $2^{\lceil \log_2 n \rceil}$, and thus, the discriminator is not shift-invariant.

\section{The case $p \geq 5$}
\label{four}

We leave the case of $p = 3$ to the next section, since the case of $p \geq 5$ is more straightforward and contains some results that are used for $p = 3$ as well. 

In this section we show that for any prime $p \geq 5$, there are no sequences of the form $(q(n))_{n \geq 0} = (\alpha n^2 + \beta + \gamma)_{n \geq 0}$ with $\alpha \neq 0$, whose discriminator is $D_q (n) = p^{\lceil \log_p n \rceil}$ for all $n \geq 1$. 

\subsection{Integer quadratic coefficients}

We begin by considering sequences of the form $(\qp(n))_{n \geq 0} = (c(p^k n^2 + bn))_{n \geq 0}$, for integers $k$, $p$, $c$, and $b$, where $p$ is prime, $p^k \geq 5$, and $\gcd (b, p) = 1$. From Eq.~\eqref{quaddiff}, we see that for all $i$ and $j$, 
\begin{equation*}
	\qp(j) - \qp(i) = (j - i)(\alpha(i + j) + \beta) = c(j - i)(p^k(i + j) + b).
\end{equation*}

First, we present a lemma concerning the factor of $p^k (i + j) + b$ in the equation for $\qp (j) - \qp (i)$. 

\begin{lemma}
	Let $p$, $k$, $b$, $r$ be integers such that $p$ and $r$ are prime, $p^k \geq 5$, $\gcd (p, b) = 1$, $r > |b|$, and $r \equiv \modd{-b} {p^k}$. Then $z = \frac{(p^k - 1)r - b}{p^k}$ is the smallest non-negative integer such that $p^k z + b \equiv \modd{0} {r}$.
	\label{zcounter}
\end{lemma}
\begin{proof}
	Since $r \equiv \modd{-b} {p^k}$ and $\gcd (b, p) = 1$, it follows that $\gcd (r, p) = 1$. Now, for the equation $p^k z + b \equiv \modd{0} {r}$, it is clear that $z \equiv \modd{(-b)(p^k)^{-1}} {r}$, and thus, there is only one solution of $z$ in the range $0 \leq z < r$. We now show that this single solution is $z = \frac{(p^k - 1)r - b}{p^k}$. 
	
	First, we observe that for $z = \frac{(p^k - 1)r - b}{p^k}$, the numerator is $p^k r - (r + b) \equiv \modd {0} {p^k}$, since $r \equiv \modd{-b} {p^k}$, thus ensuring that $z$ is an integer. It is also clear that if $z = \frac{(p^k - 1)r - b}{p^k}$, then $p^k z + b \equiv (p^k - 1)r \equiv \modd{0} {r}$. 
	
	Now, since $p^k \geq 5$, it follows that $(p^k - 1) r \geq r$, which further implies $(p^k - 1)r - b > 0$ since $r > |b|$. Furthermore, it is clear that $r + b > 0$ and thus, $p^k r - (r + b) < p^k r$. In other words, we have $0 < z < \frac{p^k r}{p^k} = r$. Hence the result follows.
\end{proof}

We now consider primes whose base-$b$ representation, for some base $b$, have specified prefixes and suffixes. We let $\Sigma_b$ denote the alphabet $\{0, 1, \ldots, b - 1\}$. The notation $[x]_b$ refers to the number that would be written as the string $x$ in base-$b$. 
\begin{lemma}
	Let $b \geq 2$ be an integer and let $x$ and $y$ be finite strings in $\Sigma_b^*$ such that $\gcd(b, [y]_b) = 1$. Then there exist infinitely many strings $w \in \Sigma_b^*$ such that $[xwy]_b$ is prime.
	\label{primefixsufpref} 
\end{lemma}

\begin{proof}
	For an integer $n \geq 1$, let $\mathcal{P}_{x, y, b, n}$ denote the set of primes of the form $[xwy]_b$ for strings $w \in \Sigma_b^*$ such that $|wy| = n$. These numbers can be represented as $[x]_b \cdot b^n + [wy]_b$. Therefore, they are congruent to $[y]_b \bmod b$ and lie in the interval from $[x]_b \cdot b^n$ to $([x]_b + 1) \cdot b^n$ exclusive.
	
	From Dirchlet's theorem on primes in arithmetic progressions, the number of primes less or equal to $[x]_b \cdot b^n$ and congruent to $[y]_b \bmod b$, denoted by $\pi ([x]_b \cdot b^n, b, [y]_b)$, is approximated by
	\begin{equation*}
		\pi ([x]_b \cdot b^n, b, [y]_b) \approx \frac{1}{\varphi (b)} \li ([x]_b \cdot b^n),
	\end{equation*}
	where $\varphi (b)$ is Euler's totient function, and $\li (m)$ is the logarithmic integral function where $\li (m) = \int_2^m dt/ \log t \approx \frac{m} {\log m}$. Therefore,
	\begin{align*}
		|\mathcal{P}_{x, y, b, n}| &= \pi (([x]_b + 1) \cdot b^n, b, [y]_b) - \pi ([x]_b \cdot b^n, b, [y]_b)\\ 
		&\approx \frac{1}{\varphi (b)} (\li (([x]_b + 1) \cdot b^n) - \li ([x]_b \cdot b^n)) \\
		&\approx \frac{1}{\varphi (b)} \left(\frac{([x]_b + 1) \cdot b^n} {\log ([x]_b + 1) + \log b^n} - \frac{[x]_b \cdot b^n} {\log [x]_b + \log b^n}\right)\\
		&\approx  \frac{([x]_b + 1) \cdot b^n - [x]_b \cdot b^n}{\varphi (b) (\log [x]_b + n \log b)}\\ 
		&= \frac{b^n}{\varphi (b) (\log [x]_b + n \log b)}.
	\end{align*}
	As $n$ grows large, this value approaches $\frac{b^n}{\varphi(b) n \log b}$. The error term for the approximation is known to be bounded by $O([x]_b \cdot b^n e^{-c \lambda ([x]_b \cdot b^n)})$, where $\lambda (m) = (\log m)^{3/5} (\log \log m)^{-1/5}$.
	\begin{align*}
		\lambda ([x]_b \cdot b^n) &= (\log [x]_b \cdot b^n)^{3/5} (\log \log [x]_b \cdot b^n)^{-1/5}\\ 
		&= (\log [x]_b + n \log b)^{3/5} (\log (\log [x]_b + n \log b))^{-1/5}\\
		&= \left(\frac{(\log [x]_b + n \log b)^3}  {\log (\log [x]_b + n \log b)}\right)^{1/5}.\\
	\end{align*}
	As $n$ grows large, we have
	\begin{align*}
		\lambda ([x]_b \cdot b^n) &= \left(\frac{(n \log b)^3}  {\log (n \log b)}\right)^{1/5},\\
		\implies [x]_b \cdot b^n \cdot \exp (-c \lambda ([x]_b \cdot b^n)) &= \frac{b^n}{\frac{1} {[x]_b} \exp (c \left(\frac{(n \log b)^3}  {\log (n \log b)}\right)^{1/5})}.
	\end{align*}
	Here, the denominator of the error term is an exponential function, with growth rate in $\Omega (e^{cn^{2/5}})$. This grows much faster than the denominator of $\frac{b^n}{\varphi(b) n \log b}$, i.e., $\varphi (b) n \log b$, which grows only linearly with $n$. In other words, as $n$ grows large, the upper bound of the error term grows much slower than the approximation $|\mathcal{P}_{x, y, b, n}| \approx \frac{b^n}{\varphi(b) n \log b}$. 
	
	In other words, we have $|\mathcal{P}_{x, y, b, n}| > 0$ for $n$ sufficiently large, i.e. $\mathcal{P}_{x, y, b, n}$ is non-empty. As there are infinitely many sets $\mathcal{P}_{x, y, b, n}$ with $n$ large enough, it follows that there are infinitely many primes of the form $[xwy]_b$ for strings $w \in \Sigma_b^*$.
\end{proof}

We use this lemma to prove the following lemma:
\begin{lemma}
	Let $p$, $k$, $c$, and $b$ be integers such that $p$ is prime, $p \geq 3$, $p^k \geq 5$, and $\gcd (p, b) = 1$. Then for all sequences of the form $(\qp(n))_{n \geq 0} = (c(p^k n^2 + bn))_{n \geq 0}$, there exists a pair of integers, $r$ and $\ell$ such that $r < p^{\ell + 1}$ and $r \nmid \qp (j) - \qp (i)$ for all $i$ and $j$ in the range $0 \leq i < j \leq p^\ell$. 
	\label{qpupperbound}
\end{lemma}
\begin{proof}
	Here, we consider the prime numbers such that their first digit in base $p^k$ is $p^k - 1$ and their last digit is equivalent to $-b$ modulo $p^k$. From Lemma~\ref{primefixsufpref}, we know that there are infinitely many such primes that fulfil these conditions, so we choose $r$ to be any of these primes such that $r > \max(|b|, |c|)$. Since the first digit of $r$ in base $p^k$ is $p^k - 1$, this means that there is an integer $u$ such that $(p^k - 1)(p^k)^u \leq r < (p^k)^{u + 1}$. Let $\ell = k(u + 1) - 1$ so that $(p^k)^{u + 1} = p^{\ell + 1}$ and $(p^k)^u = p^{\ell - k + 1}$. Therefore, we have $(p^k - 1) p^{\ell - k + 1} \leq r < p^{\ell + 1}$.
	
	We now show that $r \nmid \qp (j) - \qp (i) = c(j - i)(p^k(i + j) + b)$ for all $i$ and $j$ in the range $0 \leq i < j \leq p^\ell$. It is clear that $r \nmid c$ since $r > |c|$, and that $r \nmid (j - i)$ since $j - i \leq j \leq p^\ell < r$. Thus, it suffices to show that $r \nmid p^k(i + j) + b$ to prove $r \nmid \qp (j) - \qp (i)$.
	
	By contradiction, let us assume that there is some $i$ and $j$ in the range $0 \leq i < j \leq p^\ell$ such that $r | p^k(i + j) + b$. This implies that $p^k (i + j) + b \equiv \modd{0} {r}$. Since $i + j$ must be non-negative, we can apply Lemma~\ref{zcounter} to show that $i + j \geq \frac{(p^k - 1)r - b}{p^k}$. Noting that $r > |b|$, $r \geq (p^k - 1)p^{\ell - k + 1}$, and $p^k \geq 5$, we can deduce that 
	\begin{align*}
		i + j &\geq \frac{(p^k - 1)r - b}{p^k} = \frac{p^k r - r - b}{p^k} \geq \frac{p^k r - r - r}{p^k} = \frac{(p^k - 2) r}{p^k},\\
		&\geq \frac{(p^k - 2) (p^k - 1)p^{\ell - k + 1}}{p^k} = \frac{p^{\ell + 1} (p^{2k} - 3p^k + 2)}{p^{2k}} \geq \frac{p^{\ell + 1} (p^{2k} - 3p^k)}{p^{2k}}\\
		&= p^{\ell + 1} \left(1 - \frac{3}{p^k} \right).
	\end{align*}
	Now, if $p = 3$, then we have $p^k \geq 9$, and so
	\begin{equation*}
		i + j \geq p^{\ell + 1} \left(1 - \frac{3}{p^k} \right) \geq 3p^\ell \left(1 - \frac{3}{9} \right) = 2p^\ell.
	\end{equation*}
	Otherwise, if $p \geq 5$, then
	\begin{equation*}
		i + j \geq p^{\ell + 1} \left(1 - \frac{3}{p^k} \right) \geq 5p^\ell \left(1 - \frac{3}{5}\right) = 2p^\ell.
	\end{equation*}
	In both cases, we have $i + j \geq 2p^{\ell}$, which is a contradiction since $0 \leq i < j \leq p^\ell$. It follows that for all $i$ and $j$ in the range $0 \leq i < j \leq p^\ell$, we have $r \nmid \qp (j) - \qp (i)$.
\end{proof}

Although this section is about the case of $p \geq 5$, the proof for Lemma~\ref{qpupperbound} includes the case of $p = 3$, which is relevant to the next section.

We now show that the discriminator for $(\qp(n))_{n \geq 0}$ is not characterized by $p^{\lceil \log_p n \rceil}$. 

\begin{lemma}
	Let $p$, $k$, $c$, and $b$ be integers such that $p$ is prime, $p \geq 3$, $p^k \geq 5$, and $\gcd (p, b) = 1$. Then, for every sequence of the form $(\qp(n))_{n \geq 0} = (c(p^k n^2 + bn))_{n \geq 0}$, there exists at least one value of $n \geq 1$ such that $D_{\qp} (n) \neq p^{\lceil \log_p n \rceil}$.
	\label{notqpdisc}
\end{lemma}
\begin{proof}
	From Lemma~\ref{qpupperbound}, we know there exists a pair of integers $r$ and $\ell$ such that $r < p^{\ell + 1}$ and $r \nmid \qp (j) - \qp (i)$ for all $i$ and $j$ in the range $0 \leq i < j \leq p^\ell$. This implies that $r$ discriminates the first $p^\ell + 1$ terms of $(\qp(n))_{n \geq 0}$.  Therefore, $D_{\qp} (p^\ell + 1) \leq r$. But since $r < p^{\ell + 1}$, it follows that $D_{\qp} (p^\ell + 1) < p^{\ell + 1} = p^{\lceil \log_p (p^\ell + 1) \rceil}$. Thus, for $n =  p^\ell + 1$, we have $D_{\qp} (n) \neq p^{\lceil \log_p n \rceil}$.
\end{proof}

With Lemma~\ref{notqpdisc}, along with Lemma~\ref{generalNOTdisc}, we can show that there are no quadratic sequences with integer coefficients with discriminator $p^{\lceil \log_p n \rceil}$ for primes $p \geq 5$.
\begin{theorem}
	Let $p \geq 5$ be a prime number. Then for every quadratic sequence with integer coefficients, denoted by $(q(n))_{n \geq 0} = (\alpha n^2 + \beta + \gamma)_{n \geq 0}$ with $\alpha \neq 0$, there exists a value of $n \geq 1$ such that $D_q (n) \neq p^{\lceil \log_p n \rceil}$.
\end{theorem}
\begin{proof}
	We can apply Lemma~\ref{generalNOTdisc} to show that there exists a value of $n \geq 1$ such that $D_q (n) \neq p^{\lceil \log_p n \rceil}$ if any of the following conditions are satisfied:
	\begin{enumerate}
		\item $p \nmid \alpha$;
		\item $p | \beta$;
		\item $\alpha = p^k c$ for some integer $c$ such that $c \nmid \beta$.
	\end{enumerate}
	If neither of these conditions are satisfied, then it follows that $(q(n))_{n \geq 0}$ is of the form $(\qp(n))_{n \geq 0} = (c(p^k n^2 + bn))_{n \geq 0}$ for $k \geq 1$, and $\gcd (p, b) = 1$. We can then apply Lemma~\ref{notqpdisc} to show that $D_q (n) \neq p^{\lceil \log_p n \rceil}$ for some $n \geq 1$.
\end{proof}

This means that for primes $p \geq 5$, there are no integer-valued quadratic sequences of the form $(q(n))_{n \geq 0} = (\alpha n^2 + \beta + \gamma)_{n \geq 0}$ such that $D_q (n) = p^{\lceil \log_p n \rceil}$ for all $n \geq 1$.

\subsection{Half-integer quadratic coefficients}
We further show that for the general case of $(\qr (n))_{n \geq 0} = \left(\frac{\alpha'}{2} n^2 + \frac{\beta'}{2} n + \gamma\right)_{n \geq 0}$, there are still no such sequences with $D_{\qr} (n) = p^{\lceil \log_p n \rceil}$ for all $n \geq 1$ for any prime $p \geq 3$. Recall from Eq.~\eqref{qrdiff} that for all integers $i$ and $j$, we have
\begin{equation*}
	\qr (j) - \qr (i) = (j - i)(\alpha(i + j) + \beta) = \frac{(j - i)(\alpha' (i + j) + \beta')}{2}. \tag{\ref{qrdiff} revisited}
\end{equation*}

\begin{lemma}
	Let $p \geq 3$ be a prime number and let $(\qr (n))_{n \geq 0} = \left(\frac{\alpha'}{2} n^2 + \frac{\beta'}{2} n + \gamma\right)_{n \geq 0}$ for odd integers $\alpha'$, $\beta'$, and $\gamma$ be a quadratic sequence such that $\alpha'$ and $\beta'$ are odd, and any of the following conditions are satisfied:
	\begin{enumerate}
		\item $\alpha' = \beta'$;
		\item $p \nmid \alpha'$;
		\item $p | \beta'$;
		\item $\alpha' = p^k c$ for some integer $c$ such that $c \nmid \beta'$;
		\item $\alpha' = p^k c$ for some integer $c$ such that $c | \beta'$ and $p^k \geq 5$.
	\end{enumerate}
	Then there exists a value of $n > 0$ such that $D_{\qr} (n) \neq p^{\lceil \log_p n \rceil}$.
	\label{qrNOTdiscpee}
\end{lemma}

\begin{proof}
	If $\alpha' = \beta'$, then it was shown in Theorem~\ref{qrdisctwo} that the discriminator is $D_{\qr} (n) = 2^{\lceil \log_2 n \rceil}$. Thus the discriminator does not take values other than powers of 2. Otherwise, the middle three conditions are shown by Lemma~\ref{qrNOTdisc}.
	
	For the final condition, we can express $(\qr(n))_{n \geq 0} = (\frac{1} {2} \qp (n))_{n \geq 0} = (\frac{1} {2} c(p^k n^2 + bn))_{n \geq 0}$, where $k \geq 1$ and $b$ and $c$ are integers with $\gcd (p, b) = 1$. From Lemma~\ref{qpupperbound}, we know that there exists a pair of integers $r$ and $\ell$ such that $r < p^{\ell + 1}$ and $r \nmid \qp (j) - \qp (i)$ for all $i$ and $j$ such that $0 \leq i < j \leq p^\ell$. It follows that $r \nmid \frac{1} {2} (\qp(j) - \qp(i))$ and so, $D_{\qr} (p^\ell + 1) \leq r < p^{\ell + 1} = (p^\ell + 1)^{\lceil \log_p (p^\ell + 1) \rceil}$, and so, $D_{\qr} (n) \neq p^{\lceil \log_p n \rceil}$ for $n = p^\ell + 1$.
\end{proof}

\begin{corollary}
	Let $p \geq 5$ be a prime number. Then for all quadratic sequences of the form $(\qr (n))_{n \geq 0} = \left(\frac{\alpha'}{2} n^2 + \frac{\beta'}{2} n + \gamma\right)_{n \geq 0}$ for integers odd integers $\alpha'$ and $\beta'$, and any integer $\gamma$, there exists a value of $n \geq 1$ such that $D_{\qr} (n) \neq p^{\lceil \log_p n \rceil}$.
\end{corollary}
\begin{proof}
	For $p \geq 5$, all possible cases are covered by Theorem~\ref{qrNOTdiscpee}.
\end{proof}

\section{The case $p = 3$}
\label{five}

We finally turn to quadratic sequences with discriminator $3^{\lceil \log_3 n \rceil}$. In this section, we present a set of necessary conditions and a set of sufficient conditions for a quadratic sequence with integer coefficients to have discriminator $3^{\lceil \log_3 n \rceil}$. 

\subsection{Necessary conditions with integer quadratic coefficients}

We begin with the set of necessary conditions, as described in the following theorem.

\begin{theorem}
	Let $(q(n))_{n \geq 0} = (\alpha n^2 + \beta + \gamma)_{n \geq 0}$ be a quadratic sequence with integer coefficients such that $D_q (n) = 3^{\lceil \log_3 n \rceil}$ for all $n \geq 1$. Then there exist integers $b$ and $c$ such that $\alpha = 3c$, $\beta = bc$, and $3 \nmid bc$. Furthermore, if $b$ is even, then $c$ is also even.
	\label{qtnecessary}
\end{theorem}
\begin{proof}
	From Lemma~\ref{generalNOTdisc}, we know that $D_q (n) \neq 3^{\lceil \log_3 n \rceil}$ for some $n \geq 1$ if certain conditions are satisfied. Violating the conditions of Lemma~\ref{generalNOTdisc} implies that $p \nmid \beta$ and $\alpha = p^k c$ for integers $k$ and $c$ such that $k \geq 1$ and $c | \beta$. It follows that $D_q (n) = 3^{\lceil \log_3 n \rceil}$ implies that $(q(n))_{n \geq 0} = (p^k cn^2 + bcn)_{n \geq 0}$ where $3 \nmid bc$. 
	
	Now, from Lemma~\ref{notqpdisc}, we know that if $p^k \geq 5$, then $D_q (n) \neq p^{\lceil \log_p n \rceil}$ for some $n \geq 1$. Therefore, we must have $3^k = p^k < 5$ in order for $D_q (n) \neq 3^{\lceil \log_3 n \rceil}$. Since $k \geq 1$, it follows that $k = 1$ and thus, $\alpha = 3c$. 
	
	Finally, if $(q(n))_{n \geq 0} = (3 cn^2 + bcn)_{n \geq 0}$, then $q(0) = 3c(0) + bc(0) = 0$ and $q(1) = 3c(1) + bc(1) = c(3 + b)$. If $b$ is even, then $3 + b$ is odd. If $c$ is also odd, then $q(1)$ is odd, which means that the number 2 discriminates $\{q(0), q(1)\}$, and so, $D_q (2) = 2 \neq 3$, which contradicts $D_q (2) = 3^{\lceil \log_3 (2) \rceil} = 3$. Therefore, if $b$ is even, then $c$ must also be even. 
\end{proof}

These conditions are not sufficient, however. For example, the discriminator of the first four terms of $(3n^2 + 7n)_{n \geq 0}$ is 7 instead of $3^{\lceil \log_3 4 \rceil} = 9$, even though the necessary conditions are fulfilled. 

\subsection{Sufficient conditions with integer quadratic coefficients}

We now derive a set of sufficient conditions by considering the class of sequences of the form $(\qt(n))_{n \geq 0} = (3cn^2 + bcn)_{n \geq 0}$. Provided that $b$ and $c$ satisfy certain restrictions, we show that the discriminator sequence is $D_{\qt} (n) = 3^{\lceil \log_3 n \rceil}$. Some examples of such sequences are $(3n^2 + n)_{n \geq 0}$, $(6n^2 - 4n)_{n \geq 0}$, and $(21n^2 + 49n)_{n \geq 0}$.

Applying Lemma~\eqref{quaddiff} to sequences of the form $(\qt(n))_{n \geq 0} = (3cn^2 + bcn)_{n \geq 0}$ yields
\begin{equation}
	\qt (j) - \qt (i) = (j - i)(\alpha(i + j) + \beta) = (j - i)(3c(i + j) + bc) = c(j - i)(3(i + j) + b).
\end{equation}
Before proving any results relating to the discriminator of such sequences, we first establish the following general lemmas:
\begin{lemma}
	For all positive integers $\ell$ and $k$ such that $\ell < 2(3^k)$, there exists a pair of integers $i$ and $j$ such that $i + j = \ell$ and $0 \leq i < j \leq 3^k$.
	\label{eyejayell}
\end{lemma}
\begin{proof}
	If $\ell \leq 3^k$, choose $i = 0$ and $j = \ell$. Otherwise, if $\ell > 3^k$, choose $i = \ell - 3^k$ and $j = 3^k$. In this case, we have $i < j$ since $\ell < 2(3^k)$.
\end{proof}

\begin{lemma}
	Let $u$ and $v$ be integers such that $u \geq 3$, $v \geq 5$ and $v$ is odd, and let $k\geq 2$ be such that $3^k \leq uv < 3^{k + 1}$. Then $u + v - 1 \leq 3^k$.
	\label{youvee}
\end{lemma}

\begin{proof}
	
	We leave the finitely many cases of $k = 2$ to the reader. Otherwise, if $k \geq 3$, we have $3^k \geq 27$. There are two cases here:
	\begin{description}
		\item[Case 1: $u = 3$.] Since $v$ is odd, we have $v \leq 3^k - 2$. Therefore, $u + v - 1 \leq 3 + 3^k - 2 - 1 = 3^k$.
		\item[Case 2: $u \geq 4$.] Since $v \geq 5$, it follows that $(u - 4)(v - 4) \geq 0$. This implies that
		\begin{equation*}
			u + v \leq \frac{uv + 16}{4} < \frac{3^{k + 1} + 16}{4} \leq 3^k \cdot \frac{3}{4} + 4.
		\end{equation*}
		Since $3^k \geq 27$, we have $3^k \cdot \frac{3}{4} + 4 \leq 3^k + 1$ and thus, $u + v - 1 \leq 3^k$.
	\end{description}
\end{proof}

We now present the following lemma which enforces a set of constraints on the values of $b$ and $c$ in order to prove the lower bound of $D_{\qt} (n)$ for $n \geq 1$. 
\begin{lemma}
	Let $k \geq 0$ and let $(\qt (n))_{n \geq 0} = (3cn^2 + bn)_{n \geq 0}$ be a quadratic sequence such that $b$ and $c$ are non-zero integers that satisfy all of the following conditions:
	\begin{enumerate}
		\item $b \geq -2$.
		\item $3 \nmid bc$.
		\item If $b$ is even, then $c$ is also even.
		\item If $b$ is odd and there exists an integer $x$ such that $2(3^k) < 2^x < 3^{k + 1}$, $2^x \leq |b|$, and $(b \bmod 2^x) \equiv \modd{0} {3}$, then $c$ is even.
		\item For every prime number $p$ such that $2(3^k) < p < 3^{k + 1}$, $p \leq |b|$, and $(b \bmod p) \equiv \modd{0} {3}$, we have $p | c$. 
	\end{enumerate} 
	Then for all positive integers $m < 3^{k + 1}$, there exists integers $i$ and $j$ such that $0 \leq i < j \leq 3^k$ and $m | \qt (j) - \qt (i)$.
	\label{qtlowerbound}
\end{lemma}

\begin{proof}
	Let $m = 2^x 3^y r$ for $x \geq 0$, $y \geq 0$, $2 \nmid r$ and $3 \nmid r$.
	
	There are several cases to consider for the value of $m$.
	\begin{enumerate}
		\item $m \leq 3^k$. Choose $i = 0$ and $j = m$ so that $m | (j - i)|\qt(j) - \qt(i)$.
		
		\item $m$ is a positive power of 2. We can split this further into three cases.
		\begin{enumerate}
			\item $m = 2^x < 2(3^k)$. Let $\ell = -b(3)^{-1} \bmod 2^x$, so that $m | (3 \ell + b)$. If $\ell \neq 0$, then apply Lemma~\ref{eyejayell} so that $i + j = \ell$. Otherwise, if $\ell = 0$, then apply Lemma~\ref{eyejayell} so that $i + j = 2^x$. In both cases, we have $m | \qt (j) - \qt (i)$ and $0 \leq i < j \leq 3^k$.
			\item $m = 2^x \geq 2(3^k)$, $c$ is even. The case of $m = 2$ is trivial. Otherwise, for $m > 2$, let $\ell = -b(3)^{-1} \bmod 2^{x - 1}$. Note that $m | 2(3 \ell + b)$. If $\ell \neq 0$, then apply Lemma~\ref{eyejayell} so that $i + j = \ell$. Otherwise, if $\ell = 0$, then apply Lemma~\ref{eyejayell} so that $i + j = 2^{x - 1}$. In both cases, we have $m | \qt (j) - \qt (i)$ and $0 \leq i < j \leq 3^k$.
			
			\item $m = 2^x \geq 2(3^k)$, $c$ is odd. Note that from conditions 3 and 4 of the lemma statement, we have $b$ odd and $(b \bmod 2^x) \not\equiv \modd{0} {3}$. For $m = 2$, choose $i = 0$ and $j = 1$ so that $3(i + j) + b$ is even. Otherwise, we have $m > 2$. If $b < 0$, let $z = b$. Otherwise, let $z = b \bmod 2^x$. Then it suffices to have $2^x | 3(i + j) + z$ so that $m | \qt (j) - \qt (i)$. Note that $z$ must be odd since $b$ is odd. If $2^x \equiv \modd{z} {3}$, then let $\ell = \frac{2^x - z}{3}$. Otherwise, if $2^x \neq \modd{z} {3}$, we have $2^{x + 1} \equiv \modd{z} {3}$ since $2^x \neq \modd{0} {3}$ and $z \neq \modd{0} {3}$ , so we let $\ell = \frac{2^{x + 1} - z}{3}$. For both cases, note that $\ell > 0$ since $z$ is odd, and also that $z \geq -1$. Therefore, 
			\begin{align*}
				\ell &\leq \frac{2^{x + 1} - z}{3} \leq \frac{2^{x + 1} + 1}{3} = \frac{2(2^x) + 1}{3} \leq \frac{2(3^{k + 1} - 1) + 1}{3}\\
				&= \frac{2(3^{k + 1} - 1)}{3} < \frac{2(3^{k + 1})}{3} = 2(3^k).
			\end{align*}
			Thus, we can apply Lemma~\ref{eyejayell} so that $i + j = \ell$ and so, we have $m | \qt (j) - \qt (i)$ and $0 \leq i < j \leq 3^k$.
		\end{enumerate}
		
		\item $m$ is a prime $\geq 5$ or twice such a prime. We also have three cases here.
		\begin{enumerate}
			\item $m$ is prime and $5 \leq m < 2(3^k)$. Let $\ell = -b(3)^{-1} \bmod m$ so that $m | (3 \ell + b)$. If $\ell \neq 0$, then apply Lemma~\ref{eyejayell} so that $i + j = \ell$. Otherwise, if $\ell = 0$, then apply Lemma~\ref{eyejayell} so that $i + j = m$. In both cases, we have $m | \qt (j) - \qt (i)$ and $0 \leq i < j \leq 3^k$.
			\item $m$ is prime and $m > 2(3^k)$. If $(b \bmod m) \equiv \modd{0} {3}$, we have $m | c$, and so, $m | \qt (j) - \qt (i)$ for any choice of $i$ and $j$. Otherwise, we have $(b \bmod m) \not\equiv \modd{0} {3}$. If $b < 0$, let $z = b$. Otherwise, let $z = b \bmod m$. Then it suffices to have $m | (3(i + j) + z)$ so that $m | \qt (j) - \qt (i)$. If $m \equiv \modd{z} {3}$, then let $\ell = \frac{m - z}{3}$. Otherwise, if $m \neq \modd{0} {3}$ and $z \neq \modd{0} {3}$, then we have $2m \equiv \modd{z} {3}$, so we let $\ell = \frac{2m - z}{3}$. For either case, since $-2 \leq z < m$, we have $\ell \leq \frac{2m - z}{3} \leq \frac{2m + 2}{3}$ and that $m < 3^{k + 1} - 1$, since $m$ is odd. Therefore, 
			\begin{equation*}
				\ell \leq \frac{2m + 2}{3} = \frac{2(m + 1)}{3} < \frac{2(3^{k + 1} - 1 + 1)}{3} = \frac{2(3^{k + 1})}{3} = 2(3^k).
			\end{equation*}
			Thus, we can apply Lemma~\ref{eyejayell} so that $i + j = \ell$ and so, we have $m | \qt (j) - \qt (i)$ and $0 \leq i < j \leq 3^k$.
			
			\item $m = 2p$ for some prime $p \geq 5$. Let $\ell = -b(3)^{-1} \bmod p$ so that $p | 3 \ell + b$. We observe that 
			\begin{equation*}
				\ell < p = \frac{m}{2} < \frac{2m}{3} < \frac{2(3^{k + 1})}{3} = 2(3^k).
			\end{equation*}
			Now, if $\ell = 0$, we apply Lemma~\ref{eyejayell} so that $i + j = p$. Otherwise, if $\ell \neq 0$, we apply Lemma~\ref{eyejayell} so that $i + j = \ell$. For either case, we have $p | 3(i + j) + b$. If $b$ is even, then $c$ is also even. Otherwise, if $b$ is odd, then either $(j - i)$ or $(3(i + j) + b)$ is even. Therefore, $m = 2p | \qt (j) - \qt (i)$. 
		\end{enumerate}
		
		\item $m$ does not have any prime factors except 2 and 3. We split this into four cases.
		\begin{enumerate}
			\item $m = 3^y$ or $m = 2(3^y)$. Clearly $y \leq k$. Choose $i = 0$ and $j = 3^y$. If $b$ is odd, then $3(i + j) + b$ is even. Otherwise, $c$ is even. Therefore, $m | \qt(3^y) - \qt (0)$.
			
			\item $m = 4(3^y)$, $b$ is even. Clearly $y < k$. Choose $i = 0$ and $j = 2(3^y) < 3^{y + 1} \leq 3^k$. Since $b$ is even, $c$ must also be even, and so $m | c(j - i)$.
			
			\item $m = 2^x 3^y$, $x \geq 3$, and $b$ is even. Choose $i = 0$ and $j = 2^{x - 2} 3^y = \frac{m}{4} < \frac{3^{k + 1}}{4} < 3^k$. Then $j - i = 2^{x - 2} 3^y$ while both $c$ and $3(i + j) + b$ are even. Therefore, $m | \qt(j) - \qt (i)$.
			
			\item $m = 2^x 3^y$, $x \geq 2$, and $b$ is odd. Clearly $y < k$. Let $\ell = (-b(3)^{-1} - 3^y) \bmod 2^x$ so that $2^x | 3 (\ell + 3^y) + b$. Note that $\ell$ is even. If $\ell = 0$, choose $i = 2^{x - 1}$. Otherwise, choose $i = \frac{\ell}{2}$. In both cases, choose $j = i + 3^y$ so that $j - i = 3^y$ while $i + j = 2i + 3^y \equiv \modd{\ell + 3^y} {2^x}$, and therefore, $m | (j - i)(3 (i + j) + b) = \qt (j) - \qt (i)$. To verify that $j \leq 3^k$, note that $i \leq 2^{x - 1}$ and so, $j \leq 2^{x - 1} + 3^y$. Since $m = 2^x 3^y < 3^{k + 1}$, it follows that $2^{x - 1} \leq \frac{m}{6} < \frac{3^{k + 1}}{6} = \frac{3^k}{2}$ and so, $j \leq 2^{x - 1} + 3^y < \frac{3^k}{2} + 3^{k - 1} = \frac{5(3^k)}{6} < 3^k$.
		\end{enumerate}
		
		\item For all other possible cases of $m$, we can write $m = uv$ such that $u \geq 3$, $v \geq 5$, and $\gcd (v, 6) = 1$. In this case, let $\ell = (-b(3)^{-1} - u) \bmod v$ where $(3)^{-1}$ is the multiplicative inverse of $3$ modulo $v$, so that $v | 3(\ell + u) + b$. If $\ell$ is even, then choose $i = \frac{\ell}{2}$. Otherwise, if $\ell$ is odd, choose $i = \frac{\ell + v}{2}$. In both cases, choose $j = i + u$ so that $j - i = u$ and 
		\begin{equation*}
			3(i + j) + b = 3(2i + j) + b \equiv 3(\ell + u) + b \equiv \modd{0} {v}.
		\end{equation*}
		Therefore, $m = uv | (j - i)(3(i + j) + b) | \qt (j) - \qt (i)$. To verify that $j \leq 3^k$, note that $0 \leq i < v$ and so, $j \leq u + v - 1$. Thus, we can apply Lemma~\ref{youvee} to show that $j \leq u + v - 1 \leq 3^k$.
	\end{enumerate}
	In all cases, we have $m | \qt (j) - \qt (i)$ for some $i$ and $j$ in the required range.
\end{proof}

Before we move on to the upper bound, we provide some insight on the choices of conditions in Lemma~\ref{qtlowerbound}. 

\begin{itemize}
	\item Condition 4 is due to the argument in Case 2c, where $m$ is a power of 2 such that $2(3^k) < m < 3^{k+1}$, e.g., $m = 64$. Since $b$ is odd, it is impossible for 2 to divide both $(j - i)$ and $(3(i + j) + b)$, so we need $m | c(3(i + j) + b)$. If $z \not\equiv \modd{0} {3}$ as described in Case 2c, then there is no problem. Otherwise, the smallest non-negative solution for $3 \ell + b \equiv \modd{0} {m}$ can have $\ell$ being anywhere from $0$ to $m - 1$, which might be bigger than $2(3^k)$, thus making it impossible for $\ell = i + j$ in some cases. So we need $c$ to be even to allow $m | c(3(i + j) + b)$ then.
	\item Likewise, Condition 5 is due to case 3b, through similar logic. An example that illustrates the need for both Conditions 4 and 5 is $(3n^2 + 75n)_{n \geq 0}$, where the discriminator of the first nineteen terms is 61 (a prime), while the discriminator of the first twenty terms is 64 (a power of 2).
	\item Conditions 4 and 5 also include $2^x \leq |b|$ and $p \leq |b|$ respectively. This is because if $b$ is positive and greater than $2^x$ or $p$ respectively, then we have $z = b$, which implies $z \not\equiv \modd{0} {3}$ due to Condition 2, making Cases 2c and 3b applicable. Furthermore, for any positive $b$, there are finitely many primes $p$ such that $p \leq |b|$, so Condition 5 still ensures that $c$ is bounded.
	\item However, if $b$ is negative, there are some potential problems. This is because even if $3 \nmid b$, there are infinitely many values of $m$ such that $b \bmod m \equiv \modd{0} {3}$ while $m = 2^x$ or $p$ as described in Conditions 4 and 5, if $b$ is negative. The constraints of $2^x \leq |b|$ and $p \leq |b|$ would no longer be sufficient to capture all such scenarios. Even if we try to expand these constraints, there might be infinitely many primes $p$ for which the only solution of $3\ell + b \equiv \modd{0} {p}$ involves $\ell > 2(3^k)$, making it impossible to bound $c$ then.
	\item Despite this, the arguments in Cases 2c and 3b still work for $z \geq -2$. So if $b$ is negative, we can let $z = b$, but we add Condition 1 to ensure that $z = b \geq -2$ in such cases.
\end{itemize} 

Justification for Conditions 2 and 3 are covered by Theorem~\ref{qtnecessary}. We now move on to the upper bound on the discriminator of $(\qt (n))_{n \geq 0}$, which is handled by the following lemma.
\begin{lemma}
	Let $k \geq 0$. For all pairs of integers $i$ and $j$ satisfying $0 \leq i < j < 3^{k + 1}$, we have $3^{k + 1} \nmid \qt (j) - \qt (i)$ if $3 \nmid bc$.
	\label{qtupperbound}
\end{lemma}
\begin{proof}
	For $\qt (j) - \qt (i) = c (j - i)(3 (i + j) + b)$, it is given that $3 \nmid c$. Since $3 \nmid b$, it follows that $3(i + j) + b \equiv b \not\equiv \modd{0} {3}$, and so, $3 \nmid 3(i + j) + b$. Therefore, any powers of 3 that divide $\qt (j) - \qt (i)$ must divide the $(j - i)$. But $j - i \leq j < 3^{k + 1}$. Therefore, $3^{k + 1} \nmid \qt (j) - \qt (i)$ for all $i$ and $j$ in the range $0 \leq i < j < 3^{k + 1}$.
\end{proof}

We now compute the discriminator of $(\qt (n))_{n \geq 0} = (3cn^2 + bn)_{n \geq 0}$ with the same conditions as Lemma~\ref{qtlowerbound}.
\begin{theorem}
	Let $(\qt (n))_{n \geq 0} = (3cn^2 + bn)_{n \geq 0}$ be a quadratic sequence such that $b$ and $c$ are non-zero integers that satisfy all of the following conditions:
	\begin{enumerate}
		\item $b \geq -2$.
		\item $3 \nmid bc$.
		\item If $b$ is even, then $c$ is also even.
		\item If $b$ is odd and there exists a pair of positive integers $x$ and $k$ such that $2(3^k) < 2^x < 3^{k + 1}$ , $2^x \leq |b|$, and $(b \bmod 2^x) \equiv \modd{0} {3}$, then $c$ is even.
		\item For every prime number $p$ such that $2(3^k) < p < 3^{k + 1}$ for a positive integer $k$, $p \leq |b|$, and $(b \bmod p) \equiv \modd{0} {3}$, we have $p | c$. 
	\end{enumerate} 
	Then the discriminator $D_{\qt}(n)$ satisfies the equation
	\begin{equation}
		D_{\qt}(n) = 3^{\lceil \log_3 n \rceil}
	\end{equation}
	for $n \geq 1$.
	\label{qtdisc}
\end{theorem}

\begin{proof}
	This follows directly from an application of Lemma~\ref{generaldisc} on $(\qt(n))_{n \geq 0}$ for $p = 3$, where the two conditions of Lemma~\ref{generaldisc} are fulfilled by Lemmas~\ref{qtlowerbound} and~\ref{qtupperbound} respectively.
\end{proof}
We presented a set of conditions for which $D_{\qt}(n) = 3^{\lceil \log_3 n \rceil}$. These conditions, however, are not necessary. A simple example to illustrate this is $(3n^2 + 25n)_{n \geq 0}$, which satisfies all five conditions except Condition 5 for $p = 19$, where $2(3^2) < 19 < 3^3$, $19 \leq |25|$, and $25 \bmod 3 \equiv \modd{0} {3}$, but $19 \nmid 1$. The discriminator is still $3^{\lceil \log_3 n \rceil}$, however. This can be shown by observing that all cases in Lemma~\ref{qtlowerbound} are still applicable, except the case of $m = 19$ with $k = 2$. But even then, we can still set $i = 6$ and $j = 7$ to get $3(i + j) + 25 = 38 \equiv \modd{0} {19}$, and so, $m | \qt (7) - \qt (6)$, satisfying the result of Lemma~\ref{qtlowerbound}. 

It remains an open problem to close the gap between the necessary and sufficient conditions in order to provide a complete characterization of all quadratic sequences with discriminator $3^{\lceil \log_3 n \rceil}$.

\subsection{Half-integer quadratic coefficients}

We briefly discuss the case of $(\qr (n))_{n \geq 0} = \left(\frac{\alpha'}{2} n^2 + \frac{\beta'}{2} n + \gamma\right)_{n \geq 0}$ for odd $\alpha'$ and $\beta'$. Recall Lemma~\ref{qrNOTdiscpee},

\begin{qrpNOTdiscrepeat}
	Let $p \geq 3$ be a prime number and let $(\qr (n))_{n \geq 0} = \left(\frac{\alpha'}{2} n^2 + \frac{\beta'}{2} n + \gamma\right)_{n \geq 0}$ for odd integers $\alpha'$, $\beta'$, and $\gamma$ be a quadratic sequence such that $\alpha'$ and $\beta'$ are odd, and any of the following conditions are satisfied:
	\begin{enumerate}
		\item $\alpha' = \beta'$;
		\item $p \nmid \alpha'$;
		\item $p | \beta'$;
		\item $\alpha' = p^k c$ for some integer $c$ such that $c \nmid \beta'$;
		\item $\alpha' = p^k c$ for some integer $c$ such that $c | \beta'$ and $p^k \geq 5$.
	\end{enumerate}
	Then there exists a value of $n > 0$ such that $D_{\qr} (n) \neq p^{\lceil \log_p n \rceil}$.
\end{qrpNOTdiscrepeat}

This covers most cases of $(\qr (n))_{n \geq 0}$ for odd $\alpha'$ and $\beta'$. The only remaining case is when $\alpha' = 3c$ for some integer $c$ such that $c | \beta'$ and $3 \nmid c$. Unlike with integer coefficients, we were unable to find any examples for which $D_{\qr} (n) = 3^{\lceil \log_3 n \rceil}$. 

\begin{conjecture}
	Let $b$ and $c$ be non-zero integers such that $3 \nmid bc$. For all sequences of the form $(\qr(n))_{n \geq 0} = (\frac{3c}{2}n^2 + \frac{bc}{2}n)_{n \geq 0}$ such that $b$ and $c$ are non-zero integers with $3 \nmid bc$, there exists at least one value of $n \geq 1$ such that $D_{\qr} (n) \neq 3^{\lceil \log_3 n \rceil}$. 
\end{conjecture}
Proving this conjecture would prove that $D_{\qr} \neq 3^{\lceil \log_3 n \rceil}$ for all sequences of the form $(\qr (n))_{n \geq 0} = \left(\frac{\alpha'}{2} n^2 + \frac{\beta'}{2} n + \gamma\right)_{n \geq 0}$ with odd $\alpha'$ and $\beta'$.

\section{Acknowledgments}

I am deeply grateful to Jeffrey Shallit, who supervised my work on discriminators, patiently guiding me through numerous challenges. In particular, he suggested the approach on how to show that there are no quadratic sequences for the $p \geq 5$ case. 

I am also grateful to Pieter Moree for introducing us to this interesting topic of discriminators, and providing some useful suggestions.

\end{document}